\documentclass{amsart}

\newtheorem{theorem}{Theorem}[section]
\newtheorem{lemma}[theorem]{Lemma}
\newtheorem{corollary}[theorem]{Corollary}

\theoremstyle{definition}

\theoremstyle{remark}
\newtheorem{remark}[theorem]{Remark}

\numberwithin{equation}{section}

\newcommand{\C}{\mathbb{C}}
\newcommand{\D}{\mathbb{D}}
\newcommand{\N}{\mathbb{N}}

\newcommand{\T}{\mathbb{T}}
\newcommand{\Z}{\mathbb{Z}}

\newcommand{\mcB}{\mathcal{B}}
\newcommand{\mcF}{\mathcal{F}}

\newcommand{\mrmT}{\mathrm{T}}

\DeclareMathOperator{\spn}{\mathrm{sp}}
\DeclareMathOperator{\cspn}{\overline{\mathrm{sp}}}

\begin{document}

\title[The intersection of past and future]
{The intersection of past and future for multivariate stationary processes}


\author[A.\ Inoue]{Akihiko Inoue}
\address{Department of Mathematics, Hiroshima University,
Higashi-Hiroshima 739-8526, Japan}
\email{inoue100@hiroshima-u.ac.jp}

\author[Y.\ Kasahara]{Yukio Kasahara}
\address{Department of Mathematics, Hokkaido University,
Sapporo 060-0811, Japan}
\email{y-kasa@math.sci.hokudai.ac.jp}

\author[M.\ Pourahmadi]{Mohsen Pourahmadi}
\address{Department of Statistics,
Texas A\&M University,
College Station, TX 77843, USA}
\email{pourahm@stat.tamu.edu}
\thanks{Mohsen Pourahmadi was supported by the NFS grant DMS-1309586.}

\subjclass[2010]{Primary 60G10; secondary 62M10, 60G25.}




\begin{abstract}
We consider an intersection of past and future property of
multivariate stationary processes which  is the key to
deriving various representation theorems for their linear predictor coefficient matrices. 
We extend useful spectral characterizations for 
this property from univariate processes to multivariate processes.
\end{abstract}

\maketitle



\section{Introduction}\label{sec:1}

We write $\C^{m\times n}$ for the set of all complex $m\times n$
matrices.
Let $\{X(k):k\in\Z\}$ be a $\C^{q\times 1}$-valued, centered,
weakly stationary process, defined on a probability space $(\Omega, \mcF, P)$, which
we shall simply call a {\em $q$-variate stationary process}.
Write $X(k)=(X_1(k),\dots,X_q(k))^{\mrmT}$, and
let $M$ be the complex Hilbert space
spanned by all the entries $\{X_j(k): k\in\Z,\ j=1,\dots,q\}$ in $L^2(\Omega, \mcF, P)$,
which has inner product $(Y_1, Y_2)_{M}:=E[Y_1\overline{Y_2}]$ and
norm $\Vert Y\Vert_{M}:=(Y,Y)_{M}^{1/2}$.
For $I\subset \Z$ such as $\{n\}$,
$(-\infty,n]:=\{n,n-1,\dots\}$, $[n,\infty):=\{n,n+1,\dots\}$,
and $[m,n]:=\{m,\dots,n\}$ with $m\le n$,
we define the closed subspace $M_I^X$ of $M$ by
\[
M_I^X:=\cspn \{X_j(k): j=1,\dots,q,\ k\in I\}.
\]
Notice that $M_{[n,n]}^X=M_{\{n\}}^X=\spn\{X_1(n),\dots,X_q(n)\}$.

In this paper, we are concerned with the following {\em intersection of past and future}
property of a $q$-variate stationary process $\{X(k)\}$:
\begin{equation}
M_{(-\infty,-1]}^X\cap M_{[-n,\infty)}^X=M_{[-n,-1]}^X,\qquad n=1,2,3,\dots.
\tag{\rm IPF}
\end{equation}
It is shown in \cite[Theorem 3.1]{I1} that a univariate stationary process satisfies
(IPF) if it is purely nondeterministic (PND) (see Section \ref{sec:2} below) and
has spectral density $w$ such that
$w^{-1}$ is integrable.
We prove a multivariate analog of this sufficient condition for (IPF).
More precisely,
we show that a $q$-variate stationary process $\{X(k)\}$ satisfies
(IPF) if $\{X(k)\}$ has maximal rank (see Section \ref{sec:2} below)
and has spectral density $w$ such that
$w^{-1}$ is integrable (see Corollary \ref{cor:B} below).
We remark that such a process $\{X(k)\}$ is PND.

The importance of (IPF) for univariate stationary processes is that it,
combined with von Neumann's Alternating
Projection Theorem (cf.\ \cite[\S 9.6.3]{P}), allows one to derive explicit and useful representations of
finite-past prediction error variances (\cite{I1, I2, IK1}),
finite-past predictor coefficients (\cite{IK2}), and partial autocorrelations 
or Verblunsky coefficients
(\cite{I3, BIK, KB}), of $\{X(k)\}$.
We can extend this approach introduced by \cite{I1} to multivariate stationary processes.
In so doing, the sufficient condition for (IPF) stated above plays a crucial role. 
In our subsequent work, under the (IPF) condition and using an argument 
which involves the Alternating Projection Theorem, we extend 
various known univariate representations for the finite-past prediction error variances, 
finite-past predictor
coefficients, and partial autocorrelations to the multivariate setting.

The property (IPF) is closely related to the property
\begin{equation}
M_{(-\infty,-1]}^X\cap M_{[0,\infty)}^X=\{0\},
\tag{\rm CND}
\end{equation}
 called {\em complete nondeterminism} by Sarason \cite{S}.
Pointing out that the essence of a spectral characterization of
CND processes had been given by
Levinson and McKean \cite{LM},
Bloomfield et al.~\cite{BJH} considered various characterizations of univariate CND
processes.
For univariate stationary processes, the equivalence
(CND) $\Leftrightarrow$ (PND) $+$ (IPF) holds (see \cite[Theorem 2.3]{IK2}).
For $q$-variate processes, this equivalence is not necessarily true (see Remark \ref{rem:cnd665}
below). The main theorem of this paper is the equivalence between (IPF) and (CND) and their 
spectral characterizations similar to the univariate ones stated above, under the assumption that 
$\{X(k)\}$ is PND and has maximal rank (see Theorems \ref{thm:A} below). 
We prove the above sufficient condition for (IPF) that 
$w^{-1}$ is integrable as a simple corollary of this theorem. 
We also show an example of $\{X(k)\}$ with (IPF) for which $w^{-1}$ is not integrable, 
as another corollary of this theorem.


\section{Preliminaries}\label{sec:2}

As stated in Section \ref{sec:1}, let $\C^{m\times n}$ be the set of all complex $m\times n$
matrices, and $I_n$ the $n\times n$ unit matrix.
For $A\in \C^{m\times n}$, we denote by $A^{\mrmT}$ the transpose of $A$, and by
$\bar{A}$ and
$A^*$ the complex and Hermitian conjugates of $A$, respectively.
Thus $A^*:=\bar{A}^{\mrmT}$.

Let $\T$ be the unit circle in $\C$, i.e.,
$\T:=\{z\in\C :\vert z\vert=1\}$.
We write $\sigma$ for the normalized Lebesgue measure $d\theta/(2\pi)$ on
$([-\pi,\pi), \mcB([-\pi,\pi)))$, where $\mcB([-\pi,\pi))$ is the Borel $\sigma$-algebra
of $[-\pi,\pi)$. Thus we have $\sigma([-\pi,\pi))=1$.
For $p\in [1,\infty)$,
we write $L_p(\T)$ for the Lebesgue space of measurable functions $f:\T\to\C$
such that $\Vert f\Vert_p<\infty$, where
\[
\Vert f\Vert_p:=\left\{\int_{-\pi}^{\pi}\vert f(e^{i\theta})\vert^p\sigma(d\theta)\right\}^{1/p}.
\]
Let $L_p^{m\times n}(\T)$ be the space of $\C^{m\times n}$-valued functions on
$\T$ whose entries belong to $L_p(\T)$.

For $p\in [1,\infty)$, the Hardy class $H_p(\T)$ on $\T$ is the closed subspace of
$L_p(\T)$ defined by
\[
H_p(\T):=\left\{f\in: L_p(\T): \int_{-\pi}^{\pi}e^{im\theta}f(e^{i\theta})\sigma(d\theta)=0\
\mbox{for}\ m=1,2,\dots\right\}.
\]
Let $H_p^{m\times n}(\T)$ be the space of $\C^{m\times n}$-valued functions on
$\T$ whose entries belong to $H_p(\T)$.
Let $\D$ be the unit open disk in $\C$, i.e.,
$\D:=\{z\in\C : \vert z\vert<1\}$.
For $p\in [1,\infty)$, we write $H_p(\D)$ for the Hardy class on $\D$, consisting of
holomorphic functions $f$ on $\D$ such that
\[
\sup_{r\in [0,1)}\int_{-\pi}^{\pi}\vert f(re^{i\theta})\vert^p\sigma(d\theta)<\infty.
\]
As usual, we identify each function $f$ in $H_p(\D)$ with its boundary function
\[
f(e^{i\theta}):=\lim_{r\uparrow 1}f(re^{i\theta})\quad\mbox{$\sigma$-a.e.}
\]
in $H_p(\T)$ (cf.\ Rosenblum and Rovnyak \cite{RR}).

A function $h$ in $H_2^{n\times n}(\T)$ is called {\em outer}\/ if $\det h$ is a
$\C$-valued outer function, that is, $\det h$ satisfies
\begin{equation}
\log\vert \det h(0)\vert
=\int_{-\pi}^{\pi}\log\vert \det h(e^{i\theta})\vert \sigma(d\theta)
\label{eq:outer666}
\end{equation}
(cf.\ Katsnelson and Kirstein \cite[Definition 3.1]{KK}).

Let $\{X(k)\}$ be a $q$-variate stationary process.
If there exists a nonnegative $q\times q$ Hermitian matrix-valued function $w$ on $\T$, satisfying
$w\in L^{q\times q}_1(\T)$ and
\[
E[X(m)X(n)^*]=\int_{-\pi}^{\pi}e^{-i(m-n)\theta}w(e^{i\theta})\sigma(d\theta),\qquad n,m\in\Z,
\]
then we call $w$ the {\em spectral density}\/ of $\{X(k)\}$.
We say that $\{X(k)\}$ has {\em maximal rank}\/ if
\[
\mbox{$\{X(k)\}$ has spectral density $w$ such
that $\det w(e^{i\theta})>0$ $\sigma$-a.e.}
\tag{\rm MR}
\]
(see Rozanov \cite[pp.\ 71--72]{R}).
A $q$-variate stationary process $\{X(k)\}$ is said to be {\em purely nondeterministic}\/ (PND) if
\[
\cap_{n\in\Z}M_{(-\infty,n]}^X=\{0\}.
\tag{\rm PND}
\]
Every PND process $\{X(k)\}$ has spectral density but it does not necessarily have maximal rank
unlike univariate processes (see \cite[Theorem 4.1]{R}).
So we combine the two to define the condition
\[
\mbox{$\{X(k)\}$ satisfies both (MR) and (PND)}.
\tag{\rm A}
\]
A necessary and sufficient condition for (A) is that $\{X(k)\}$
has spectral density $w$ such that $\log \det w\in L_1(\T)$ (see \cite[Theorem 6.1]{R}).

Let $\{X(k)\}$ be a $q$-variate stationary process satisfying (A),
and let $w$ be its spectral density.
Then, the spectral density $w$ of $\{X(k)\}$ has a
decomposition of the form
\begin{equation}
w(e^{i\theta})=h(e^{i\theta})h(e^{i\theta})^*\quad \mbox{$\sigma$-a.e.}
\label{eq:decomp111}
\end{equation}
for some outer function $h$ in $H_2^{q\times q}(\T)$, and $h$ is unique up to a constant unitary
factor (see, e.g., \cite[Chapter II]{R} and Helson and Lowdenslager \cite[Theorem 11]{HL}).

\begin{lemma}\label{lem:lind332}
We assume {\rm (A)}\/. Then,
$X_j(k)$, $k\in\Z$, $j=1,\dots,q$, are linearly independent.
\end{lemma}

\begin{proof}
Let $h(z)=\sum_{n=0}^{\infty}c(n)z^n$, $z\in\D$,
be the power series expansion of $h$, where $\{c(n)\}_{n=0}^{\infty}$ is a $\C^{q\times q}$-valued
sequence whose entries $\{c_{i, j}(n)\}_{n=0}^{\infty}$, $i,j=1,\dots,q$, belong to
$\ell^2$.
Then, there exists a $q$-variate stationary process $\{\xi(k)\}$, called the {\em innovation
process}\/ of $\{X(k)\}$, satisfying $E[\xi(n)\xi(m)^*]=\delta_{n, m}I_q$ and
\begin{align*}
& X(n)=\sum_{k=-\infty}^{n}c(n-k)\xi(k),\quad n\in\Z,\\
& M_{(-\infty,n]}^X=M_{(-\infty,n]}^{\xi},\quad n\in\Z,
\end{align*}
where $M_{(-\infty,n]}^{\xi}:=\cspn\{\xi_j(k): k\le 0,\ j=1,\dots,q\}$ in $L^2(\Omega, \mcF,P)$
(see Theorem 4.3 in \cite[Chapter II]{R}).

Suppose $\sum_{k=m}^na(k)X(k)=0$ for $n, m\in\Z$ with $m\le n$ and
$a(k)\in\C^{1\times q}$, $k=m,\dots,n$. Let $Q$ be the projection operator from $M$
onto the orthogonal complement $(M_{(-\infty,n-1]}^X)^{\bot}$ of $M_{(-\infty,n-1]}^X$.
Then,
\[
0=Q\left(\sum\nolimits_{k=m}^na(k)X(k)\right)=a(n)c(0)\xi(n).
\]
Since $\xi_1(n),\dots,\xi_q(n)$ are linearly independent, we have $a(n)c(0)=0$.
However, $c(0)$ is invertible by (\ref{eq:outer666}), whence $a(n)=0$. In the same way,
we also obtain $a(n-1)=\cdots=a(m)=0$. Thus, $X_j(k)$'s are linearly independent.
\end{proof}

In addition to (\ref{eq:decomp111}), $w$ has a
decomposition of the form
\begin{equation}
w(e^{i\theta})=h_{\sharp}(e^{i\theta})^*h_{\sharp}(e^{i\theta})\quad \mbox{$\sigma$-a.e.}
\label{eq:decomp222}
\end{equation}
for another outer function $h_{\sharp}$ in $H_2^{q\times q}(\T)$, and $h_{\sharp}$ is also
unique up to a constant unitary factor. In fact, for an outer function $g$ in
$H_2^{q\times q}(\T)$ satisfying $w(e^{i\theta})^{\mrmT}=g(e^{i\theta})g(e^{i\theta})^*$
$\sigma$-a.e., we may take $h_{\sharp}=g^{\mrmT}$. It should be noticed that
while we may take $h_{\sharp}=h$ for the univariate case $q=1$, there is no such simple
relation between $h$ and $h_{\sharp}$ for $q\ge 2$.

We denote by $L(w)$ the complex Hilbert space consisting of all measurable
functions $f: \T\to \C^{1\times q}$ with
$\int_{-\pi}^{\pi}f(e^{i\theta})w(e^{i\theta})f(e^{i\theta})^*\sigma(d\theta)<\infty$,
which has inner product
\[
(f,g)_w:=\int_{-\pi}^{\pi}f(e^{i\theta})w(e^{i\theta})g(e^{i\theta})^*\sigma(d\theta)
\]
and norm $\Vert f\Vert_w:=(f,f)_w^{1/2}$.
For $k\in\Z$ and $j=1,\ldots,q$, we define $e_j(k)\in L(w)$ by
\[
e_j(k)(z)
:=(0,\dots,0,z^{-k},0,\dots,0),\quad z\in\T,
\]
where $z^{-k}$ is in the $j$-th coordinate.
For an interval $I\subset\Z$, let $L_I(w)$ be the closed subspace of $L(w)$
spanned by $\{e_j(k): k\in I,\ j=1,\dots,q\}$.
By taking $I_q$ as $w$, we regard $L^{1\times q}_2(\T)$ as the complex Hilbert space
$L(I_q)$ with inner product
$(f,g)_{I_q}:=\int_{-\pi}^{\pi}f(e^{i\theta})g(e^{i\theta})^*\sigma(d\theta)$
and norm $\Vert f\Vert_{I_q}:=(f,f)_{I_q}^{1/2}$, and $H_2^{1\times q}(\T)$
as its closed subspace.

We put, for $p\in [1,\infty)$,
\[
\overline{H_p^{1\times q}(\T)}:=\left\{\bar{f}: f\in H_p^{1\times q}(\T)\right\}.
\]

\begin{lemma}\label{lem:fundament268}
We assume {\rm (A)}\/. Then, for $n\in\mathbb{Z}$ and 
outer functions $h$ and $h_{\sharp}$ in $H_2^{q\times q}(\T)$
satisfying {\rm (\ref{eq:decomp111})}\/ and {\rm (\ref{eq:decomp222})}\/, respectively,
the following two equalities hold:
\begin{align}
&L_{(-\infty,n]}(w)=z^n\cdot H_2^{1\times q}(\T)\cdot h^{-1},
\label{eq:fund111}\\
&L_{[n,\infty)}(w)
=z^{-n}\cdot \overline{H_2^{1\times q}(\T)}\cdot (h_{\sharp}^*)^{-1}.
\label{eq:fund222}
\end{align}
\end{lemma}

\begin{proof}We prove only (\ref{eq:fund222}); one can prove (\ref{eq:fund111}) in a similar way.
Define an antilinear bijection $G:L(w)\to L^{1\times q}_2(\T)$ by
$G(f):=\overline{fh_{\sharp}^*}=\bar{f}h_{\sharp}^{\mrmT}$.
Since
\[
\Vert G(f)\Vert_{I_q}^2=\Vert fh_{\sharp}^*\Vert_{I_q}^2
=\int_{-\pi}^{\pi}f(e^{i\theta})h_{\sharp}(e^{i\theta})^*
\left\{f(e^{i\theta})h_{\sharp}(e^{i\theta})^*\right\}^*\sigma(d\theta)
=\Vert f\Vert_w^2,
\]
the map $G$ preserves the norms of $f\in L(w)$. Let
\[
\C^{1\times q}[z]:=\spn \{e_j(k): k\le 0,\ j=1,\dots,q\}
\]
be the space of polynomials with coefficients in $\C^{1\times q}$. Since $h_{\sharp}^{\mrmT}$
is also an outer function in $H_2^{q\times q}(\T)$, it follows from
the Beurling--Lax--Halmos Theorem
that $\C^{1\times q}[z]\cdot h_{\sharp}^{\mrmT}$
is dense in $H_2^{1\times q}(\T)$ (cf.\ \cite[Remark 5.6 and Theorem 5.3]{KK}). Moreover,
\[
L_{[n,\infty)}(w)=\cspn\{e_j(k): k\ge n,\ j=1,\dots,q\}
\]
and
\[
G(\spn\{e_j(k): k\ge n,\ j=1,\dots,q\})=z^n\cdot \C^{1\times q}[z]\cdot h_{\sharp}^{\mrmT}.
\]
Thus,
\[
L_{[n,\infty)}(w)=G^{-1}\left(z^n\cdot H_2^{1\times q}(\T)\right)
=z^{-n}\cdot \overline{H_2^{1\times q}(\T)}\cdot (h_{\sharp}^*)^{-1},
\]
as desired.
\end{proof}


\section{The Past and future}\label{sec:3}

For a $q$-variate stationary process $\{X(k)\}$,
the next theorem holds without (A).

\begin{theorem}\label{thm:cnd-pnd-ipf327}
A $q$-variate CND process satisfies {\rm (IPF)}\/.
\end{theorem}

\begin{proof}
For any $q$-variate stationary process $\{X(k)\}$, we have
\begin{equation}
M_{(-\infty,n]}^X=M_{(-\infty,m-1]}^X + M_{[m,n]}^X,\qquad m, n\in\Z,\ m\le n.
\label{eq:vector-sum281}
\end{equation}
For, the inclusion $\supset$ is trivial, while $M_{(-\infty,m-1]}^X$ is closed
and $M_{[m,n]}^X$ is finite-dimensional, whence $M_{(-\infty,m-1]}^X + M_{[m,n]}^X$
is also closed (see Halmos \cite[Problem 8]{H}), which implies $\subset$.

For $n\in\N$, let $x\in M_{(-\infty,-1]}^X\cap M_{[-n,\infty)}^X$.
Since $x\in M_{(-\infty,-1]}^X$, it follows from (\ref{eq:vector-sum281}) that
$x=y+z$ for some $y\in M_{(-\infty,-n-1]}^X$ and $z\in M_{[-n,-1]}^X$.
Since $x, z\in M_{[-n,\infty)}^X$, we have
\[
y=x-z\in
M_{(-\infty,-n-1]}^X\cap M_{[-n,\infty)}^X.
\]
Therefore, if $\{X(k)\}$ is CND, then $y=0$ or $x=z\in M_{[-n,-1]}^X$, so that
\[
M_{(-\infty,-1]}^X\cap M_{[-n,\infty)}^X\subset M_{[-n,-1]}^X.
\]
Since the converse inclusion $\supset$ is trivial, $\{X(k)\}$ satisfies (IPF).
\end{proof}

\begin{remark}\label{rem:cnd665}
The converse of Theorem \ref{thm:cnd-pnd-ipf327} does not hold without
additional assumptions. For example, let $\{Y(k): k\in\Z\}$ be a univariate
CND stationary process; the simplest example is a white noise. Then $\{Y(k)\}$ is PND.
Define a two-variate stationary process $\{X(k): k\in\Z\}$ by $X(k):=(Y(k-1), Y(k))^{\mrmT}$.
For $I\subset \Z$, let
$M_I^Y:=\cspn\{Y(k): k\in I\}$ in $L^2(\Omega, \mcF,P)$.
Then, for $n,m\in\Z$ with $n\le m$, we have
\[
M_{(-\infty,n]}^X=M_{(-\infty,n]}^Y,\quad
M_{[n,\infty)}^X=M_{[n-1,\infty)}^Y,\quad
M_{[n,m]}^X=M_{[n-1,m]}^Y.
\]
Since $\cap_nM_{(-\infty,n]}^X=\cap_nM_{(-\infty,n]}^Y=\{0\}$, $\{X(k)\}$ is PND.
Furthermore, for $n\ge 1$,
\[
M_{(-\infty,-1]}^X\cap M_{[-n,\infty)}^X
=M_{(-\infty,-1]}^Y\cap M_{[-n-1,\infty)}^Y
=M_{[-n-1,-1]}^Y=M_{[-n,-1]}^X,
\]
whence $\{X(k)\}$ satisfies (IPF). However,
\[
M_{(-\infty,-1]}^X\cap M_{[0,\infty)}^X=M_{(-\infty,-1]}^Y\cap M_{[-1,\infty)}^Y
=M_{\{-1\}}^Y\ne \{0\},
\]
whence $\{X(k)\}$ is not CND. Notice that $\{X(k)\}$ has the degenerate spectral
density
\[
w_X(e^{i\theta})=
\left(
\begin{matrix}
w_Y(e^{i\theta}) & e^{i\theta}w_Y(e^{i\theta})\\
e^{-i\theta}w_Y(e^{i\theta}) & w_Y(e^{i\theta})
\end{matrix}
\right),
\]
where $w_Y$ is the spectral density of $\{Y(k)\}$.
\end{remark}

We assume (A), and for outer functions $h$ and $h_{\sharp}$ in $H_2^{q\times q}(\T)$
satisfying (\ref{eq:decomp111}) and (\ref{eq:decomp222}), respectively,
we consider the following two conditions:
\begin{align}
\left\{z^{-1}\cdot \overline{H_2^{1\times q}(\T)}\cdot (h_{\sharp}^*)^{-1}\right\}\cap
\left\{H_2^{1\times q}(\T)\cdot h^{-1}\right\}&=\{(0,\dots,0)\},
\label{eq:cnd676}\\
\left\{\overline{H_2^{1\times q}(\T)}\cdot (h_{\sharp}^*)^{-1}\right\}\cap
\left\{H_2^{1\times q}(\T)\cdot h^{-1}\right\}&=\C^{1\times q}.
\label{eq:cnd787}
\end{align}
For any $a\in\C^{1\times q}$, we have $ah^*_{\sharp}\in \overline{H_2^{1\times q}(\T)}$,
$ah\in H_2^{1\times q}(\T)$ and
\[
a=ah^*_{\sharp}(h^*_{\sharp})^{-1}=ahh^{-1},
\]
whence the inclusion $\supset$  in (\ref{eq:cnd787}) always holds.

Let
$X(k)=\int_{-\pi}^{\pi}e^{-ik\theta}Z(d\theta)$, $k\in\Z$,
be the spectral representation of $\{X(k)\}$ satisfying (A), where $Z$ is the
random spectral measure such that
\[
E[Z(\Lambda_1)Z(\Lambda_2)^*]=\int_{\Lambda_1\cap \Lambda_2}w(e^{i\theta})\sigma(d\theta),
\qquad \Lambda_1, \Lambda_2\in \mcB([-\pi,\pi)).
\]
Define an isometric isomorphism $S: L(w)\to M$ by
\[
S(f):=\int_{-\pi}^{\pi}f(e^{i\theta})Z(d\theta), \qquad f\in L(w).
\]
Then, $S(e_j(k))=X_j(k)$ for $k\in\Z$ and $j=1,\dots,q$, whence we have
\begin{equation}
S(L_I(w))=M_I^X,\qquad I\subset \Z.
\label{eq:S425}
\end{equation}

\begin{lemma}\label{lem:key-lemma11}
We assume {\rm (A)}\/. Then,
the following two conditions are equivalent:
\begin{enumerate}
\item {\rm (\ref{eq:cnd676})}\/ holds.
\item $M_{(-\infty,0]}^X\cap M_{[1,\infty)}^X=\{0\}$.
\end{enumerate}
\end{lemma}

\begin{proof}
By (\ref{eq:S425}), 
(2) is equivalent to
$L_{(-\infty,0]}(w)\cap L_{[1,\infty)}(w)=\{(0,\dots,0)\}$,
which, in turn, is equivalent to (1) by Lemma \ref{lem:fundament268}.
\end{proof}

\begin{lemma}\label{lem:key-lemma}
We assume {\rm (A)}\/. Then,
the following two conditions are equivalent:
\begin{enumerate}
\item {\rm (\ref{eq:cnd787})}\/ holds.
\item $M_{(-\infty,0]}^X\cap M_{[0,\infty)}^X=M_{\{0\}}^X$.
\end{enumerate}
\end{lemma}

\begin{proof}
We have 
$L_{\{0\}}(w)=\spn\{e_j(0): j=1,\dots,q\}=\C^{1\times q}$. Hence, by (\ref{eq:S425}), 
(2) is equivalent to
$L_{(-\infty,0]}(w)\cap L_{[0,\infty)}(w)=\C^{1\times q}$,
which, in turn, is equivalent to (1) by Lemma \ref{lem:fundament268}.
\end{proof}

Here is the main theorem of this paper.

\begin{theorem}\label{thm:A}
We assume {\rm (A)}\/. Then,
the following five conditions are equivalent:
\begin{enumerate}
\item {\rm (\ref{eq:cnd676})}\/ holds.
\item {\rm (\ref{eq:cnd787})}\/ holds.
\item {\rm (CND)}\/ holds.
\item $M_{(-\infty,-1]}^X\cap M_{[-n,\infty)}^X=M_{[-n,-1]}^X$ for some $n\in\N$.
\item {\rm (IPF)}\/ holds.
\end{enumerate}
\end{theorem}

\begin{proof}
By Lemma \ref{lem:key-lemma11}, (1) and (3) are equivalent. 
By Lemma \ref{lem:key-lemma}, (2) (resp., (5)) implies (4) (resp., (2)). 
By Theorem \ref{thm:cnd-pnd-ipf327}, (3) implies (5).
Suppose (4). Then,
\begin{align*}
&M_{(-\infty,-1]}^X\cap M_{[0,\infty)}^X \subset
M_{(-\infty,-1]}^X\cap M_{[-n,\infty)}^X=M_{[-n,-1]}^X,\\
&M_{(-\infty,-1]}^X\cap M_{[0,\infty)}^X \subset
M_{(-\infty,n-1]}^X\cap M_{[0,\infty)}^X=M_{[0,n-1]}^X.
\end{align*}
However, by Lemma \ref{lem:lind332}, we have $M_{[-n,-1]}^X \cap M_{[0,n-1]}^X=\{0\}$,
whence (3).
\end{proof}

The next corollary gives a sufficient condition
for (IPF) in terms of the spectral density.

\begin{corollary}\label{cor:B}
We assume {\rm (MR)}\/ and that the spectral density $w$ of $\{X(k)\}$ satisfies
$w^{-1}\in L^{q\times q}_1(\T)$. Then $\{X(k)\}$ satisfies {\rm (IPF)}.
\end{corollary}

\begin{proof}
Since
$(w^{-1})_{j, j}=\sum_{i=1}^q\vert (h^{-1})_{i, j}\vert^2$ for $j=1,\dots,q$,
the condition $w^{-1}\in L^{q\times q}_1(\T)$ implies $h^{-1}\in L^{q\times q}_2(\T)$.
Hence, by \cite[Theorem 3.1]{KK} and \cite[Theorem 4.23]{RR}, $h^{-1}\in H^{q\times q}_2(\T)$,
so that
\[
H_2^{1\times q}(\T)\cdot h^{-1} \subset H_1^{1\times q}(\T).
\]
Similarly, we have $(h_{\sharp})^{-1}\in H^{q\times q}_2(\T)$, and
\[
\overline{H_2^{1\times q}(\T)}\cdot (h_{\sharp}^*)^{-1} \subset
\overline{H_1^{1\times q}(\T)}.
\]
However, $\overline{H_1^{1\times q}(\T)} \cap H_1^{1\times q}(\T)=\C^{1\times q}$, whence
(\ref{eq:cnd787}). Therefore, by Theorem \ref{thm:A}, $\{X(k)\}$ satisfies (IPF).
\end{proof}

\begin{remark}
A stationary process $\{X(k)\}$ is said to be {\em minimal} if $X(0)$ cannot be interpolated precisely
using all the other values of the process.
The condition  $w^{-1}\in L^{q\times q}_1(\T)$ in Theorem 3.5 is known to be necessary and sufficient
for the minimality of a stationary process. See Section 10 of \cite[Chapter II]{R}.
\end{remark}

The next corollary gives an example 
of $\{X(k)\}$ with (IPF) for which $w^{-1}$ is not integrable (compare \cite[Proposition 3]{BJH}).

\begin{corollary}\label{cor:C}
Let $B$ be an invertible matrix in $\C^{q\times q}$. 
Then $\{X(k)\}$ with spectral density $w(e^{i\theta})=\vert 1+e^{i\theta}\vert BB^*$ satisfies {\rm (IPF)}.
\end{corollary}

\begin{proof}
We can take $h=(1+z)^{1/2}B$ and $h_{\sharp}=(1+z)^{1/2}B^*$. 
Suppose that there exist $f=(f_1,\dots,f_q), g=(g_1,\dots,g_q)\in H_2^{1\times q}(\T)$ such that
\[
z^{-1}\bar{f}(h_{\sharp}^*)^{-1}=gh^{-1}.
\]
Then, since $(h_{\sharp}^*)^{-1}h=e^{i\theta/2}I_q$ for $z=e^{i\theta}\ (-\pi<\theta<\pi)$, we have
\begin{equation}
e^{-i\theta}\left\{\overline{f_j(e^{i\theta}})\right\}^2=\left\{g_j(e^{i\theta})\right\}^2, 
\qquad j=1,\dots,q.
\label{eq:fg584}
\end{equation}
From $(g_j)^2\in H_1(\T)$, we get
\begin{equation}
\int_{-\pi}^{\pi}e^{im\theta} \left\{g_j(e^{i\theta})\right\}^2 \sigma(d\theta)=0
\label{eq:F-coeff223}
\end{equation}
for $m=1,2,\dots$, while, from $(f_j)^2\in H_1(\T)$ and (\ref{eq:fg584}), 
we see that (\ref{eq:F-coeff223}) also holds for $m=0,-1,\dots$, whence 
$g_j=0$ for $j=1,\dots,q$. Thus (\ref{eq:cnd676}) holds. 
Therefore, by Theorem \ref{thm:A}, $\{X(k)\}$ satisfies (IPF).
\end{proof}


\end{document}